\documentclass[11pt]{amsart}
\usepackage{amsmath}
\usepackage{amsfonts}
\usepackage{amssymb}
\usepackage{amsthm}
\usepackage{comment}
\usepackage{epsfig}
\usepackage{psfrag}
\usepackage{mathrsfs}
\usepackage{amscd}
\usepackage[all]{xy}
\usepackage{rotating}
\usepackage{lscape}
\usepackage{amsbsy}
\usepackage{verbatim}
\usepackage{moreverb}
\usepackage{sseq}
\usepackage{pdfpages}
\usepackage{enumerate}
\usepackage{float}
\restylefloat{table}

\newtheorem{lem}{Lemma}[section]
\newtheorem{rem}[lem]{Remark}
\newtheorem{prop}[lem]{Proposition}
\newtheorem{thm}[lem]{Theorem}
\newtheorem{cor}[lem]{Corollary}

\newtheorem{defn}[lem]{Definition}
\newtheorem{ex}[lem]{Example}

\theoremstyle{definition}

\theoremstyle{remark}

\DeclareMathOperator{\al}{\alpha}

\DeclareMathOperator{\holim}{holim}

\DeclareMathOperator{\Fun}{Fun}

\DeclareMathOperator{\C}{\mathbb{C}}

\DeclareMathOperator{\Ce}{{\sc C}_{\eta}}

\newcommand{\be}{\beta}

\newcommand{\bQ}{{\mathbb{Q}}}

\renewcommand{\to}{\longrightarrow}

\newcommand{\fib}{\twoheadrightarrow}

\def\mapright#1.{\buildrel #1 \over \longrightarrow}
\def\sc#1{{\mathcal #1}}
\def \C  {\sc C}
\def \D {\sc D}
\def \P {\sc P}

\def\Ce{{\sc C}^{\eta}}

\def \DGA {\sc D \sc G \sc A}

\def\Cen#1{({\sc C}^{\eta})^{\times #1}}

\renewcommand{\bQ}{{\mathbb{Q}}}

\def\Fun  {\operatorname{Fun}}

\def\hfib{\operatorname{hofiber}}
\def\fib{\operatorname{hofib}}
\def\ifiber{\operatorname{ifiber}}
\def\Fun{\operatorname{Fun}}

\def\holim{\operatorname{holim}}
\def\hocolim{\operatorname{hocolim}}

\input xy
\xyoption{all}
\xyoption{frame}

\title{Unbased Calculus for Functors to Chain Complexes}

\author[Basterra]{Maria Basterra}
\address{Department of Mathematics \& Statistics, University of New Hampshire}

\author[Bauer]{Kristine Bauer }
\address{Department of Mathematics \& Statistics, University of Calgary}

\author[Beaudry]{Agn\`es Beaudry}
\address{Department of Mathematics, University of Chicago}

\author[Eldred]{Rosona Eldred}
\address{Department of Mathematics, Universit\"at M\"unster}

\author[Johnson]{Brenda Johnson}
\address{Department of Mathematics, Union College}

\author[Merling]{Mona Merling}
\address{Department of Mathematics, Johns Hopkins University}

\author[Yeakel]{Sarah Yeakel}
\address{Department of Mathematics, University of Illinois at Urbana-Champaign}

\begin{document}
\maketitle
\thispagestyle{plain}
\pagestyle{plain}
\section {Introduction}
In a series of papers published between 1990 and 2003, Tom Goodwillie developed what is now known as the calculus of homotopy functors.  The calculus of homotopy functors associates to a given functor of spaces or spectra $F$,  a so-called Taylor tower of functors and natural transformations, 
\[
\xymatrix{ && F \ar[dl]\ar[d]\ar[dr] \\
\cdots \ar[r] & P_{n+1}F \ar[r] & P_nF \ar[r] & P_{n-1}F \ar[r]  \cdots \ar[r] & P_1F \ar[r] & P_0F,}
\] resembling the Taylor series for functions of  real variables.  In particular, Goodwillie's theory produces a universal $n$-excisive approximation to a homotopy functor $F$ (\cite{G1}, \cite{G2}, \cite{G3}).  Inspired by Goodwillie's work, the fifth author and Randy McCarthy produced a related theory of calculus in an abelian setting which produces what can be thought of as a ``discrete" Taylor tower for a functor (\cite{JM}).  We write $\Gamma_nF$ for the $n$th term of the discrete tower of a functor $F$. 

In the Johnson-McCarthy discrete calculus, a homotopy functor is approximated by a universal {\it degree $n$} functor.  While $n$-excisive functors are necessarily degree $n$, the converse does not generally hold.  The Johnson-McCarthy model was originally developed  for use in algebraic settings, and functors were assumed to be from a pointed category to an abelian category (often, chain complexes).   The hypothesis that the domain category be pointed was more restrictive than what Goodwillie's theory required; nonetheless, the pointed theory has been quite useful.  In particular,
it has been used  successfully to express certain interesting homology theories as derivatives of naturally arising functors.  For example,
Johnson and McCarthy, and Kantorovitz and McCarthy have provided ways of viewing  Andr\'e-Quillen homology as parts of discrete calculus towers (\cite{JM}, \cite{KM}). 


Recently, the Johnson-McCarthy theory of calculus was expanded to include functors from categories that are not necessarily pointed to categories that are not necessarily abelian (\cite{BJM}), more specifically functors from a simplicial model category to a pointed stable simplicial model category.   
At the {\it Women in Topology} workshop, our team endeavored to lay the groundwork for extending some discrete calculus computations of Kantorovitz and McCarthy to the unbased setting.  As a first step, we had to establish the existence of the unbased calculus in the context we needed, namely, for functors to a category of chain complexes of modules over a fixed commutative ring.  




Much of the construction of the Taylor tower in \cite{BJM} carries over readily to this context.   However, one of the essential and most delicate steps in the construction consists in proving that a particular functor $t$ is part of a cotriple.  For this,  one needs to prove  that  certain identities hold up to isomorphism, rather than just up to weak equivalence (see Lemma 2.5 and Appendix A of \cite{BJM}).  In the case of \cite{BJM}, the proof required making careful use of an explicit model for homotopy limits in simplicial model categories and establishing that several key isomorphisms held for that model.   As our target category of chain complexes is not a simplicial model category, we needed to redo this part of the construction before proceeding further.   We do so in this paper.   


In revisiting the proof that $t$ is a cotriple for  functors to chain complexes, we could have tried to use a similar model for homotopy limits (by choosing a framing on our target category) and verified that the required properties could be established for this model as well.   Instead, we chose to take advantage of the fact that in the category of chain complexes, we can construct an explicit model for iterated fibers which allows us to prove directly what we need for the  analogue of $t$ in this context.


The paper is organized as follows.  In Section \ref{s:recap} we define the categories  and terminology we will be working with, and state the main result of this paper:  that for a functor $F$ from an unbased simplicial model category to chain complexes over a commutative ring, one can construct a Taylor tower in which the $n$th term, $\Gamma_nF$, is a degree $n$ approximation to $F$.  In this section, we also outline the proof of this result.      With the exception of the proof that $t$ is a cotriple, as described above, the proofs used in \cite{BJM} carry over to this context of functors  to chain complexes.   In Section \ref{s:ifibers} we describe our models for iterated fibers, starting first with an explicit model for homotopy fibers in the category of chain complexes of modules over a fixed commutative ring.   This model appears in \cite{W}  and seems to be generally well-known.   We include a proof that this construction is equivalent to the standard definition of the homotopy fiber of $f:X\rightarrow Y$ as the homotopy pullback of the diagram $X\rightarrow Y\leftarrow 0$ as this argument does not seem to be in the literature.  Also in this section we use these explicit models to provide  concrete infinite deloopings  of the first terms in our Taylor towers  when evaluated at the initial objects in their source categories.     Section \ref{s:cotriple} contains the main technical result of the paper -- we use our model for iterated fibers to define $t$ and 
 prove that $t$ is a cotriple, thereby completing the proof of the main result of section 2. 
\vskip .2 in 

\noindent{\bf Acknowledgments:}  We would like to thank the scientific committee at the Banff International Research Station for supporting the 2013 {\it Women in Topology} meeting and the Clay Mathematics Institute for a generous grant to finance travel to the workshop.  We  thank Randy McCarthy and Peter May for many helpful conversations in the development of this work.


\section{Unbased Cotriple Calculus}\label{s:recap} 

The goal of any theory of functor calculus is to approximate a homotopy functor $F$  with a tower of   functors $\{ \Gamma_nF\}_{n\geq 0}$ whose individual terms possess properties that make them easier to work with.  The tower  can be considered as a Taylor series approximation to $F$.  In our case, we seek to approximate $F$ by functors that are degree $n$.  In this section, we review the construction of \cite{BJM} for universal degree $n$ approximations to homotopy functors.  

\subsection{Preliminaries}

We work with functors whose source categories are  simplicial model categories.  Let $\sc C$ be such a category.  Unlike the models of cotriple calculus defined in \cite{JM} or \cite{MO}, we will not assume that the category $\sc C$ is based, i.e., that it has the same initial and terminal object.  We specify an initial and terminal object by selecting a morphism $\eta:A\to B$ and letting ${\sc C}^\eta$ be the category of factorizations of $\eta$.  That is, an object $X\in {\sc C}^\eta$ is a diagram $A\to X\to B$ in $\C$ whose composition is $\eta$.  

The category ${\sc C}^\eta$ has a model structure inherited from $\sc C$ as in \cite{MP}, Theorem 15.3.6 or \cite{Q}, Proposition II.2.6.  In particular, a morphism $f:X\to Y$ in ${\sc C}^\eta$ is a commuting diagram
\[ \xymatrix{A \ar[r] \ar@{=}[d] & X \ar[r] \ar[d] & B \ar@{=}[d]\\ A \ar[r] & Y \ar[r] & B,}\]
and we say that $f$ is a weak equivalence, fibration or cofibration if $f$ is so in the underlying category ${\sc C}$.  

Let $R$ be a commutative ring and ${\sc Ch(R)}$ be the category of unbounded chain complexes of $R$-modules.   Recall that there is a model category structure on ${\sc Ch(R)}$ where the weak equivalences are quasi-isomorphisms, the fibrations are levelwise surjections, and $i:A\rightarrow B$ is a cofibration if it is a dimensionwise split inclusion with cofibrant cokernel (\cite{Hovey}, \S 2.3).   An object $M_*$ in ${\sc Ch(R)}$ is cofibrant provided that each entry is projective and any map from $M_*$ to an exact complex is chain homotopic to $0$, i.e., if it is DG-projective (\cite{Hovey}, \S 2.3; \cite{Hovey2}, \S 2.1).   We note that ${\sc Ch(R)}$ is a proper model category (\cite{Hovey3}, Theorem 1.7).  

When constructing the degree $n$ approximation of a functor $F$, we will assume that $F$ is a functor from ${\sc C}^\eta$ to ${\sc Ch(R)}$ that preserves weak equivalences.  In \cite{BJM}, the target category was taken to be a good category of spectra, such as symmetric spectra \cite{HSS} or the $S$-modules of \cite{EKMM}.  The category ${\sc Ch(R)}$ shares several desirable properties with these categories of spectra.  In particular, ${\sc Ch(R)}$ is a stable category (in the sense that square diagrams are homotopy cocartesian if and only if they are homotopy cartesian - see, e.g., \cite{W} for the case $R={\bQ}$) and there are nice models for homotopy limits and colimits in this category.  These properties are sufficient for most of the construction of the discrete calculus, with the exception of the proof that $t$ is a cotriple, as discussed in the introduction.  


\subsection{Cross effects and cotriples.}  Underlying the construction of the functor $\Gamma_nF$ is the notion of the $n$th cross effect of $F$, $cr_nF$, a functor of $n$ variables that measures in some sense the failure of $F$ to be ``additive."   We define the cross effects below.   
We use $\Fun(\Ce,{\sc Ch(R)})$ to denote the category of functors from $\Ce$ to ${\sc Ch(R)}$.  We use ${\P}({\bf n})$ to denote the power set of ${\bf n}=\{1, 2, \dots, n\}$.  We treat this as a category whose morphisms are given by set inclusions.  By an $n$-cubical diagram in a category $\D$, we mean a functor from ${\P}({\bf n})$ to $\D$. 

Cross effects are defined as iterated homotopy fibers of certain $n$-cubical diagrams in ${\sc Ch(R)}$.   The category ${\sc Ch(R)}$ is proper so by  the homotopy fiber of a map $f:P\rightarrow Q$, we mean the homotopy limit of the diagram\[
P\rightarrow Q\leftarrow 0.
\]
 By the iterated homotopy fiber of an $n$-cubical  diagram ${\mathcal X}$ in ${\sc Ch(R)}$, which we denote $\ifiber({\mathcal X})$, we mean the object obtained by first taking homotopy fibers of all maps in one direction, i.e., in the direction  determined by set inclusions $S\subseteq S\cup\{i\}$ with $S\cap \{i\}=\emptyset$, then taking homotopy fibers of these homotopy fibers in a second direction $j\neq i$, and continuing in this fashion until all independent directions have been exhausted.     For the explicit model of the iterated homotopy fiber used in this paper, see Definition \ref{d:ifiber} and Lemma \ref{l:ifiber}.   We use ${\rm ifiber}$ to denote the iterated homotopy fiber of an $n$-cube.
 
 \begin{defn}\label{d:cross}
Let $H$ be a functor in $\Fun(\Ce,{\sc Ch(R)})$, $G$ be a functor in $\Fun(\Cen n,{\sc Ch(R)})$, and ${\bf X}=(X_1, X_2,\dots, X_n)$ be an $n$-tuple  of objects in the category $\Ce$. 
\begin{itemize}

\item Let 
$G^{\bf X}: \P ({\bf n}) \to {\sc Ch(R)}$ 
be the $n$-cubical diagram defined for $U\in \P ({\bf n})$ by
\[G^{\bf X}(U)=G({\bf X}(U))\]
where ${\bf X}(U)$ is the $n$-tuple
$(Z_1(U), Z_2(U), \dots, Z_n(U))
$
with
\[Z_i(U)=\begin{cases}X_i&\text{if\ }i\notin U,\\
B &\text{if\ }i\in U.\\
\end{cases}\]
\item The functor ${t}G: \Cen n\rightarrow {\sc Ch(R)}$ assigns to ${\bf X}=(X_1, \dots, X_n)$ the iterated homotopy fiber of $G^{\bf X}$.
\item Let $\sqcup_n:\Cen n\rightarrow \Ce$ be the functor  that takes the $n$-tuple \\ $(X_1, X_2,\dots, X_n)$ to the coproduct over $A$, $(X_1\amalg_A \cdots \amalg_A X_n)$.  
Associated to ${\bf X}$ is the square diagram $\sqcup_n^{\bf X}: {\P}({\bf n})\rightarrow \Ce$.   The $n$th cross effect of $H$ is the functor $cr_nH: \Cen n\rightarrow{\sc Ch(R)}$ given by 
\[cr_nH(X_1, X_2, \dots, X_n):=\ifiber H(\sqcup_n^{\bf X})={ t}(H\circ\sqcup_n)({\bf X}).
\]

\item  The functor $\perp_n: \Fun(\Ce,{\sc Ch(R)})\rightarrow \Fun(\Ce,{\sc Ch(R)})$ is obtained by precomposing $cr_n$ with the diagonal.   That is, for an object $X$ in $\Ce$, $\perp_nH(X)=cr_nH(X,X,\dots, X)$.   
\end{itemize}
\end{defn}

\begin{rem}\label{n:hinvariant}

\begin{enumerate}
\item[]
\item The assignment above $G \mapsto tG$ defines an endofunctor on the category $\Fun(\Cen n,{\sc Ch(R)})$ which we will denote by $t$ below.
\item The cross effects of a homotopy invariant functor are homotopy invariant when evaluated on cofibrant objects. 
\end{enumerate}
\end{rem}

\begin{ex}
In the case $n=2$, $cr_2H(X_1, X_2)$ is the iterated homotopy fiber of the square diagram
\[
\xymatrix{H(X_1\coprod_A X_2)\ar[r]\ar[d]&H(B\coprod _A X_2)\ar[d]\\
H(X_1\coprod _A B)\ar[r]&H(B\coprod _A B).}
\]
\end{ex}

The construction of $\Gamma_nF$ in \cite{BJM} depends on establishing that $\perp_n$ is a cotriple on $\Fun(\Ce,{\sc Ch(R)})$.   
To prove  that $\perp_n$ forms a cotriple,  a sequence of adjunctions is constructed
\begin{equation}\label{e:adj}
 \xymatrix{ {\Fun}({\Ce}, {\sc Ch(R)}) \ar@/^/[r]^{\sqcup^n} &{\Fun}({\Cen n}, {\sc Ch(R)}) \ar@/^/[l]^{\Delta^*}\ar@/^/[r]^{t^+} & {\Fun}({\Cen n}, {\sc Ch(R)})_{\bf t}\ar@/^/[l]^{U^+}}
\end{equation}
whose components we explain below.  We begin with the pair on the left.

\begin{defn}  \label{d:sqcup}Define a functor 
\[\Delta^*: \Fun({\Cen n}, {\sc Ch(R)}) \to \Fun({\Ce}, {\sc Ch(R)})\]
on $H:\Cen n\rightarrow {\sc Ch(R)}$  by $\Delta^*H(X) = H(X, \ldots,  X)$.  Let 
\[ \sqcup ^n: \Fun({\Ce}, {\sc Ch(R)}) \to \Fun({\Cen n}, {\sc Ch(R)}) \]
be the functor defined by precomposition with the functor $\sqcup_n$ of Definition \ref{d:cross}.  That is, for a functor $F$  
\[ \sqcup ^n(F)(X_1, X_2, \dots, X_n) = F(X_1\amalg_A  \cdots \amalg_A X_n) .\] 
\end{defn}
The next result is Proposition 3.16 of \cite{BJM}.   The proof can be found there.  
\begin{prop} The functors $\Delta^*$ and $\sqcup ^n$ are an adjoint pair of functors, with $\Delta^*$ being the left adjoint.

\end{prop}

The adjoint pair on the right side of (\ref{e:adj}) involves the functor ${t}$ of Definition \ref{d:cross}.   We establish in Section 4 that there are natural transformations $\xi: t\rightarrow tt$ and $\gamma: t\rightarrow {\rm id}$ that make  $(t,\xi, \gamma)$ a cotriple for functors  from ${\Cen n}$ to ${\sc Ch(R)}$.

Categories equipped with a cotriple $T$ have an associated category of $T$-coalgebras, related to the original category by a  forgetful-cofree adjunction (see \cite{Mac}, \S VI).
For our cotriple $t$, the category of ${t}$-coalgebras, ${\Fun}({\Cen n}, {\sc Ch(R)})_{\bf t}$, consists of functors $G:{\Cen n}\to {\sc Ch(R)}$ that are equipped with a section $\beta:G\to {t}G$ of the natural transformation 
$\gamma_G:{t}G\to G$, that make the diagram
\[ \xymatrix{ G \ar[r]^{\beta} \ar[d]_{\beta} & {t}G\ar[d]^{{t}\beta} \\ {t}G\ar[r]_{\xi_G} & {tt}G}\]
commute.  For example, for any functor $G\in{\Fun}({\Cen n},{\sc Ch(R)})$, there is an associated ${t}$-coalgebra $({t}G, \xi_G)$.     Let $t^+: {\Fun}({\Cen n}, {\sc Ch(R)})\to  {\Fun}({\Cen n}, {\sc Ch(R)})_{\bf t}$ be the free coalgebra functor, which is defined on objects by $t^+(G)=({t}G, \xi_G)$.  

\begin{thm}  The functors
\[  \xymatrix{ {\Fun}({\Cen n}, {\sc Ch(R)}) \ar@/^/[r]^{t^+} & {\Fun}({\Cen n}, {\sc Ch(R)})_{\bf t}\ar@/^/[l]^{U^+}}\]
are an adjoint pair of functors, with the forgetful functor $U^+$ being the left adjoint.
\end{thm}

This is Theorem 3.14 of \cite{BJM}.  

Hence, from (\ref{e:adj}) one obtains another adjoint pair by composition:
\begin{equation}
 \xymatrix{ {\Fun}({\Ce}, {\sc Ch(R)}) \ar@/^/[rr]^{t^+\circ \sqcup^n} &&{\Fun}({\Cen n}, {\sc Ch(R)})_{\bf t}\ar@/^/[ll]^{\Delta^* \circ U^+}.}
\end{equation}

 When evaluated on a functor $H$, the composition of the  top  arrow with $U^+$ gives $U^+\circ t^+\circ \sqcup^n(H)= cr_nH$ as in Definition 1.1.  Thus, $\perp _n$ is the composition of the  left adjoint $\Delta ^*\circ U^+$ with the right adjoint
${t}^+\circ\sqcup ^n$.  As a composition of adjoints, $\perp_n$ is part of a cotriple (see \cite{We}, Appendix A.6).   The coproduct functor $\sqcup^n(F)$ is not a homotopy functor, even if $F$ is a homotopy functor, unless it is evaluated on cofibrant objects $(X_1, \ldots , X_n)$.   In order for $\perp_nH$ to be homotopy invariant, we precompose with a cofibrant replacement functor.  Henceforth,  $\perp_nH(X)$ means $\perp_nH(X^{cof})$ where $X^{cof}$ is a functorial cofibrant replacement of $X$.

In particular, the counit for the adjunction produced by the pair $(\Delta ^*\circ U^+, {t}^+\circ\sqcup _n)$
yields a natural transformation $ \epsilon:\perp _n\rightarrow {\rm id}$.   And, a natural 
transformation $\delta : \perp _n\rightarrow \perp _n\perp _n$ is defined by $\Delta ^*\circ U^+(\iota 
_{{t}^+\circ\sqcup ^n})$ where $\iota$ is a unit for the adjunction.   This gives us the following.

\begin{thm} The functor and natural transformations $$(\perp_n,
\delta : \perp _n\rightarrow \perp _n\perp _n,  \epsilon:\perp _n\rightarrow {\rm id})$$  form a cotriple on the category of functors $\Fun(\Ce,{\sc Ch(R)})$.
\end{thm}

Every cotriple yields a simplicial object (see \cite{We}, 8.6.4) whose face and degeneracy maps are induced by the counit and comultiplication of the cotriple.  Let $\perp_n^{*+1} F(X)$ denote the simplicial chain complex arising from the cotriple $\perp_n$ in $\Fun(\Ce, {\sc Ch(R)})$.  This is a functor from the simplicial category $\Delta^{op}$ of finite sets and order preserving maps to chain complexes.  As is the convention in \cite{BJM}, we let $|\perp_n^{*+1}F(X)|$ denote the ``fat" realization of this simplicial chain complex.  That is, 
\[ |\perp_n^{*+1} F(X)| :\ = \operatorname{hocolim}_{\Delta^{op}} \perp_n^{*+1} F(X),\]
where we assume that $\operatorname{hocolim}$ is a homotopy invariant functor (see the discussion in \cite{Hirschhorn}, Chapter 19).  When $\perp_n^{*+1}F(X)$ is cofibrant, the fat realization is weakly equivalent to the usual geometric realization.

\subsection {The degree $n$ approximation of a functor}

In this section we will approximate homotopy functors $F:\Ce\to {\sc Ch(R)}$ by functors that satisfy a kind of higher additivity condition called the {\it degree} of the functor.

\begin{defn}\label{d:degree} A functor $F:\Ce\to {\sc Ch(R)}$ is degree $n$ if \[cr_{n+1}F(X_1, \ldots, X_{n+1})\simeq 0\] for all $(n+1)$-tuples $(X_1, \ldots, X_{n+1})$.  
\end{defn}

In order to approximate $F:\Ce\to {\sc Ch(R)}$ by a degree $n$ functor, we would like to eliminate the failure of $F$ to be degree $n$.  As this information is contained in the (iterated) cross-effects, we make the following definition.

\begin{defn}(Definition 5.3, \cite{BJM})  The $n$th term in the cotriple Taylor tower of $F$  is the functor
\[ \Gamma_nF :\ = \operatorname{hocofiber} \left( | \perp^{*+1}_{n+1} F| \stackrel{\hat{\epsilon}}{\longrightarrow} F\right),\]
where $\hat{\epsilon}$ is the composition 
of the natural simplicial map from $\perp_{n+1}^{*+1}$ to the simplicial object $id^{*+1}$  induced by the counit $\epsilon$ of the cotriple $\perp_{n+1}$, with the weak equivalence $|{\rm id}^{*+1}F| \stackrel{\simeq}{\longrightarrow} F$. 
\end{defn}

We let $p_nF:F\to \Gamma_nF$ denote the natural transformation in the resulting cofibration sequence
\[ |\perp_{n+1}^{*+1}F| \to F \to \Gamma_nF.\]

\begin{thm} For a given functor $F:\Ce\to {\sc Ch(R)}$, the functor $\Gamma_nF$ is degree $n$.
\end{thm}

The proof is the same as the proof of Proposition 5.4 of \cite{BJM}.   Although it is stated for the target category of spectra, the proof as written will apply equally well in ${\sc Ch(R)}$, as it relies only on formal properties of adjoint pairs, homotopy limits and homotopy colimits.   In particular,  
 the proof involves applying a general fact about adjoint pairs of functors (Lemma 5.5 of \cite{BJM}) to the adjoint pair $(\Delta^*\circ U^+, t^+\circ \sqcup_n)$, and then using the fact that   finite homotopy limits commute with finite and filtered homotopy colimits.

One can also construct natural transformations $q_n: \Gamma_nF\rightarrow \Gamma_{n-1}F$ 
as in \cite{BJM}, so that the functors $\Gamma_nF$ assemble to form a Taylor tower for $F$.  

%

\section{Iterated fibers in ${\sc Ch(R)}$}\label{s:ifibers}
Let $R$ and ${\sc Ch(R)}$ be as defined in Section 2.  In this section, we describe our models for homotopy fibers and iterated homotopy fibers of chain complexes over $R$.  
\subsection{Homotopy fibers in ${\sc Ch(R)}$}

\begin{defn}\label{d:fiber}
Let $f:U\rightarrow V$ be a map of chain complexes in ${\sc Ch(R)}$.  Let $\fib(f)$ be the chain complex with $\fib(f)_n=U_n\oplus V_{n+1}$ and $$d(u,v)=(d(u), -f(u)-d(v)).$$
\end{defn}

This construction will serve as our model for the homotopy fiber of $f$ in ${\sc Ch(R)}$.  
\begin{prop}\label{p:fiber}
Let $f:U\rightarrow V$ be a map of chain complexes in ${\sc Ch(R)}$.   Then $\fib(f)$ is weakly equivalent to the homotopy fiber of $f$.
\end{prop}  

\begin {proof}
We first replace $f$ with a fibration.  
Let $P(f)$ be the chain complex with $P(f)_n=U_n\oplus V_{n+1}\oplus V_n$ and $$d(u,v,v')=(d(u),-f(u)-d(v)+v', d(v')).$$
We define $\beta: P(f)\rightarrow V$ by $\beta(u,v,v')=v'$, and  $\alpha: U\rightarrow P(f)$ by $\alpha(u)=(u,0, f(u))$.
One  can easily see  that $\alpha$ and $\beta$ are chain maps and that $\beta$ is a fibration.  In addition, it is straightforward to prove that $\alpha$ is a weak equivalence.  Clearly, $\beta\circ\alpha=f$.   Moreover, the kernel of $\beta$ is $\fib(f)$.  As a result, $\fib(f)$ is the pullback of the diagram $0\rightarrow V\leftarrow P(f)$.  
By 13.3.7 and 13.4.4 of \cite{Hirschhorn}, this pullback is weakly equivalent to the homotopy fiber of $f.$
\end{proof}

\subsection{Iterated Fibers of $n$-cubes}   We use the model for homotopy fiber defined above to define the iterated homotopy fiber of an $n$-cubical diagram of objects in ${\sc Ch(R)}$.   

\begin{defn}\label{d:ifiber}
Let ${\mathcal X}:{\mathcal P}({\bf n})\rightarrow {\sc Ch(R)}$ be an $n$-cubical diagram of objects in ${\sc Ch(R)}$.  The iterated homotopy fiber of ${\mathcal X}$, denoted $\ifiber{\mathcal X}$, is the object in ${\sc Ch(R)}$ obtained by first taking the homotopy fibers of all maps of the form ${\mathcal X}(U\subsetneq U\cup\{1\})$, $U\subseteq \{2, \dots, n\}$ in $\mathcal X$ to obtain an $(n-1)$-cube $\widetilde{\mathcal{X}}$, then taking the homotopy fibers of all maps of the form  
${\widetilde{\mathcal{X}}}(W\subsetneq W\cup\{2\})$, $W\subseteq\{3, \dots, n\}$, in the resulting $(n-1)$-cube of homotopy fibers, and continuing in this fashion until we have taken homotopy  fibers in all $n$ independent directions from the original $n$-cube.  
\end{defn}
We illustrate this definition with an example in the case $n=2$.
\begin{ex}
For a square diagram $\sc X : \P({\bf 2}) \to {\sc Ch(R)}$
\[
 \xymatrix{\sc X (\emptyset) = A \ar[r]^f\ar[d]_{\al}&B = \sc X (\{1\}) \ar[d]^{\be}\\
\sc X (\{2\}) = C\ar[r]_g&D = \sc X (\{1,2\}),}
\]
the first step in the construction of the iterated homotopy  fiber of $\sc X$ involves taking the homotopy fibers of the maps $f$ and $g$.   This yields two chain complexes, $\fib(f)$ and  $\fib(g)$ with 
$$
\fib(f)_k=A_k\oplus B_{k+1}, \\ \  d(a,b)=(d_A(a),\  -f(a)-d_{B}(b)),
$$
and 
$$
\fib(g)_k=C_k\oplus D_{k+1}, \\ \  d(c,d)=(d_C(c),\  -g(c)-d_D(d)).$$
The maps $\{\alpha_{k}\oplus \beta_{k+1}\}$ form a chain map from $\fib(f)$ to $\fib(g)$.  The second (and final) step in constructing the iterated homotopy fiber of $\sc X$ is to determine the homotopy fiber of this chain map.   That is, 
\[
\ifiber({\sc X})=\fib\left (\vcenter{\xymatrix{\fib(f)\ar[d]_{\alpha\oplus \beta}\\\fib(g)}  } \right).
\]
Applying the definition of homotopy fiber, we see that 
\[
\ifiber({\sc X})_k=A_k\oplus B_{k+1}\oplus C_{k+1}\oplus D_{k+2}\]
with \[ d_{\rm{ifib}}(a,b,c,d)=(d_A(a),\  -f(a)-d_B(b),\  -\alpha(a)-d_C(c),\  -\beta(b)+g(c)+d_D(d)).
\]

\end{ex}

More generally, we can use induction to describe the iterated homotopy fiber of an $n$-cube.

\begin{lem}  \label{l:ifiber} For $T\subseteq {\bf n}$, $i\notin T$, let $\sigma_i^T$ be the inclusion $T\rightarrow T\cup \{i\}$.  
Let ${\sc X}$ be an $n$-cube in ${\sc Ch(R)}$.  Then in degree $k$, the iterated homotopy fiber of ${\sc X}$ is the $R$-module
\[
\ifiber({\sc X})_{k} =\bigoplus_{T\subseteq {\bf n}}{\sc X}(T)_{k+|T|}.
\]
The differential will send the term $x$ indexed by $T\subseteq {\bf n}$ to the sum of terms
\[
(-1)^{|T|}d_{{\sc X}(T)}(x)+\sum_{i\notin T}(-1)^{{\mathrm {sgn}}(\sigma_i^T)+1}{\sc X}(\sigma_i^T)(x),
\]
where \[{\mathrm {sgn}}(\sigma_i^T)=|\{s\in T\ |\ s>i\}|.\]
\end{lem}
\begin{proof}  The proof is by induction.
The cases $n=1$ and $n=2$ have already been established.  Suppose that ${\sc Y}$ is an $(n+1)$-cube.   Let ${\sc Y}_1$ be the $n$-cube obtained by restricting ${\sc Y}$ to subsets of ${\bf n}$, that is for $U\subset {\bf n}$, ${\sc Y}_1(U) = {\sc Y}(U)$.   Let ${\sc Y}_2$ be the $n$-cube defined by  ${\sc Y}_2(U)={\sc Y}(U\cup \{n+1\})$.  Then ${\sc Y}$ is a map of $n$-cubes, ${\sc Y}_1\rightarrow {\sc Y}_2$.  By definition, 
\[
\ifiber({\sc Y})=\fib\left (\ifiber({\sc Y}_1)\rightarrow \ifiber({\sc Y}_2)\right ).
\]
Assuming the result holds for $\ifiber ({\sc Y}_1)$ and $\ifiber({\sc Y}_2)$, we see that 
\begin{align*}
\ifiber({\sc Y})_k&=\ifiber({\sc Y_1})_k\oplus \ifiber({\sc Y}_2)_{k+1}\\
&=\bigoplus_{T\subseteq {\bf n}}[{\sc Y}(T)]_{k+|T|}\oplus \bigoplus _{T\subseteq {\bf n}}[{\sc Y}(T\cup \{n+1\})]_{k+1+|T|}.
\end{align*}

To see that the differential is what we claim, let $(A,B)\in \ifiber({\sc Y}_1)_k\oplus \ifiber({\sc Y}_2)_{k+1}$ and note that the differential for $\fib(\ifiber({\sc Y}_1)\rightarrow \ifiber({\sc Y}_2))$ takes this pair to 
\[
(d(A), -d(B)+\bigoplus _{T\subseteq {\bf n}}(-{\sc Y}(\sigma^T_{n+1})(a_T)))
\] 
where $a_T$ is the term in $A$ indexed by $T$.  
For $T\subseteq {\bf n}$, consider the summands in $A$ and $B$ indexed by $T$ and $T\cup \{n+1\}$, namely ${\sc Y}(T)_{k+|T|}$ and ${\sc Y}(T\cup \{n+1\})_{k+1+|T|}$.  
The differential takes $y\in {\sc Y}(T)_{k+|T|}$ to  the sum 
\begin{align*}
&(-1)^{|T|}d(y)+\left (\sum _{i\notin T, i\in {\bf n}}(-1)^{{\mathrm{sgn}}(\sigma _i^T)+1}{\sc Y}(\sigma _i^T)(y)\right )-{\sc Y}(\sigma ^T_{n+1})(y)\\&=(-1)^{|T|}d(y)+\left (\sum _{i\notin T, i\in {\bf n}}(-1)^{{\mathrm{sgn}}(\sigma _i^T)+1}{\sc Y}(\sigma _i^T)(y)\right )+(-1)^{{\mathrm{sgn}}(\sigma ^T_{n+1})+1}{\sc Y}(\sigma ^T_{n+1})(y)\\
&=(-1)^{|T|}d(y)+\sum _{j\notin T, j\in {\bf n+1}} (-1)^{{\mathrm{sgn}}(\sigma_j^T)+1}{\sc Y}(\sigma ^T_j)(y).\\
\end{align*}
Similarly, the differential takes $y\in {\sc Y}(T\cup \{n+1\})$ to 
\begin{align*}
&-\left((-1)^{|T|}d(y)+\sum_{i\notin T, i\in {\bf n}}(-1)^{{\mathrm {sgn}}(\sigma_i^T)+1}{\sc Y}(\sigma_i^{T\cup \{n+1\}})(y)\right )\\
&=(-1)^{|T\cup \{n+1\}|}d(y)+\sum_{i\notin T\cup \{n+1\}}(-1)^{{\mathrm{sgn}}(\sigma _i^{T\cup \{n+1\}})+1}{\sc Y}(\sigma_i^{T\cup \{n+1\}})(y).\\
\end{align*}
\end{proof}

\subsection{Total fibers}

Before using the preceding results about iterated fibers to prove that $t$ is a cotriple, we conclude this section with some remarks about a related notion, the total fiber, for square diagrams.  We also indicate how these facts about the total fiber can be used to obtain some information about the first term in the Taylor tower of a functor. 

\begin{defn}
For a square diagram ${\bf X}$,
\[\xymatrix{A\ar[r]^f\ar[d]_{\alpha}&B\ar[d]^{\beta}\\
C\ar[r]_g&D,}\]
the total fiber of ${\bf X}$, denoted $\operatorname{tfiber}({\bf X})$, is the homotopy fiber of the map
from $A$ to ${\operatorname{holim}(C\rightarrow D\leftarrow B)}$,  the homotopy pullback of $(C\rightarrow D\leftarrow B)$.  
\end{defn}

\begin {rem}\label{r:tfiber}As was the case with the iterated fiber, we can construct an explicit model for the total fiber using the constructions in Definition \ref{d:fiber} and Proposition \ref{p:fiber}.  In particular, when one replaces $C$ with $P(g)$ (where $P(g)$ is defined as in the proof of Proposition \ref{p:fiber}), then the homotopy pullback of 
\[
\xymatrix{&B\ar[d]^{\beta}\\
C\ar[r]_{g}&D}
\] 
is the pullback of 
\[
\xymatrix{&B\ar[d]^{\beta}\\
P(g)\ar[r]_{\hat g}&D.}
\] 

Applying the model for the homotopy fiber in Definition \ref{d:fiber} to the map from $A$ to the pullback of $(P(g)\rightarrow D\leftarrow B)$ yields an explicit model for the total fiber of the square.  In particular, in degree $k$, the total fiber consists of $5$-tuples $(a,b,c,d,d')$ in $A_k\oplus B_{k+1}\oplus C_{k+1}\oplus D_{k+2}\oplus D_{k+1}$ with $d'=\beta(b)$, and the differential takes such a tuple to 
\[
(d_A(a), -f(a)-d_B(b), -\alpha(a)-d_C(c), g(c)+d_D(d)-\beta(b), -\beta f(a)-d_D(d')).
\]
\end{rem}
It is straightforward to prove the following.

\begin{lem}\label{l:tfib=ifib}
The explicit models of Lemma \ref{l:ifiber} for the iterated fiber and Remark \ref{r:tfiber} for the total fiber yield isomorphic chain complexes. 
\end{lem}
We use this to prove the next result.  
\begin {prop}\label{p:OmegaSigma}
Let $F:{\sc C}^{\eta}\rightarrow {\sc Ch(R)}$ where ${\sc C}$ is a simplicial model category and $\eta:k\rightarrow B$ is a morphism in ${\sc C}$.  Suppose further that $F(B)\simeq 0$, i.e., that $F$ is reduced, and that $F$ is a degree $1$ functor.  Then there is a weak equivalence of chain complexes:
\[
F(k)\simeq \Omega F(B\otimes _kB).
\]
In this context $\Omega$ indicates that the chain complex has been shifted down one degree, i.e., for a chain complex $X$, $(\Omega X)_k=X_{k+1}$.  
\end{prop}

\begin{proof}
By Definitions \ref{d:cross} and \ref{d:degree}, the fact that $F$ is degree $1$ means that when it is applied to the square
\[
\xymatrix{k\otimes_kk=k\ar[r]\ar[d]&B\otimes_kk=B\ar[d]\\
k\otimes_k B=B\ar[r]&B\otimes_kB,}
\]
the iterated fiber of the resulting square
\[
\xymatrix{F(k)\ar[r]\ar[d]&F(B)\simeq 0\ar[d]\\
F(B)\simeq 0\ar[r]&F(B\otimes_kB),}
\]
is equivalent to $0$.  By Lemma \ref{l:tfib=ifib}, the total fiber is also equivalent to $0$, so that we have
\[
F(k)\simeq {\operatorname{holim}}(0\rightarrow F(B\otimes _kB)\leftarrow 0)\simeq \Omega F(B\otimes_k B).
\]
\end{proof}

If $G$ is a reduced functor, then $\Gamma_1G$ will be a reduced degree $1$ functor.   The following is an immediate consequence of the proposition.

\begin {cor}
Let $G:{\sc C}^{\eta}\rightarrow {\sc Ch(R)}$ where ${\sc C}$ is a simplicial model category and $\eta:k\rightarrow B$ is a morphism in ${\sc C}$.  Suppose further that $G$ is reduced.  Then there is a weak equivalence of chain complexes:
\[
\Gamma_1G(k)\simeq \Omega \Gamma_1G(B\otimes _kB).
\]
\end{cor}

Since $1$-excisive functors take any cocartesian square (rather than just those cocartesian squares whose initial object is $k$) to a cartesian square, a similar argument can be used to prove the next results, where $P_1G$ denotes the universal $1$-excisive approximation to $G$.  This is the first term in Goodwillie's Taylor tower for the functor $G$. (See \cite{BJM}, Definitions 4.1, 6.1, and 6.3.)   We use the following notation.

\begin{defn}
For $\eta: k\rightarrow B$ and $X$ in $\C^{\eta}$, let 
\[
\Sigma_B X=\hocolim (B\leftarrow X\rightarrow B),
\]
and for $n\geq 2$,
\[
\Sigma_B^n X=\hocolim(B\leftarrow \Sigma_B^{n-1}X\rightarrow B).
\] 
\end{defn}
 
We note that it follows directly from the definition of $P_1G$ that when $G$ is a reduced functor, $P_1G(X)$ will be an infinite loop object.  The next result shows how to realize $P_1G(X)$ as an $n$-fold loop object for any $n$.

\begin {prop}\label {p:deloop}
Let $F:{\sc C}^{\eta}\rightarrow {\sc Ch(R)}$ where ${\sc C}$ is a simplicial model category and $\eta:k\rightarrow B$ is a morphism in ${\sc C}$. Let $X$ be an object in $\C^{\eta}$.   Suppose further that $F(B)\simeq 0$, i.e., that $F$ is reduced, and that $F$ is a $1$-excisive functor (see Definition 4.1 of \cite{BJM}).  Then for all $n\geq 1$ there is a weak equivalence of chain complexes:
\[
F(X)\simeq \Omega^n F(\Sigma_B^nX).
\]
\end{prop}

\begin{proof}
Consider the cocartesian square that defines $\Sigma_B X$:
\[
\xymatrix{X\ar[r]\ar[d]&B\ar[d]\\
B\ar[r]&\Sigma_B X.}
\]
Since $F$ is $1$-excisive, applying it to this square yields a cartesian square
\[
\xymatrix{F(X)\ar[r]\ar[d]&F(B)\ar[d]\\
F(B)\ar[r]&F(\Sigma_B X).}
\]
Then, as in the  proof of Proposition \ref{p:OmegaSigma} (where $F$ was degree 1), we have
\[
F(X)\simeq \holim (F(B)\rightarrow F(\Sigma_B X)\leftarrow F(B))\simeq \Omega F(\Sigma _B X).
\]
The result follows by induction on $n$.  
\end{proof}

By noting that Theorem 6.9 of \cite{BJM} holds in this context, i.e., that $\Gamma_1F$ and $P_1F$ are weakly equivalent when evaluated at the initial object $k$ in $\C^{\eta}$, we can use this to show that $\Gamma_1F(k)$ is also an $n$-fold loop object for any $n$.  

\begin{cor}\label{c:deloop}
Let $F:{\sc C}^{\eta}\rightarrow {\sc Ch(R)}$ where ${\sc C}$ is a simplicial model category and $\eta:k\rightarrow B$ is a morphism in ${\sc C}$.   Suppose further that $F$ is reduced.  Then for all $n\geq 1$ there is a weak equivalence of chain complexes:
\[
\Gamma_1F(k)\simeq P_1 F(k)\simeq \Omega^n P_1 F(\Sigma_B^nk).
\]
\end{cor}

\section {The Cotriple $t$}\label{s:cotriple}
In this section, we revisit the functor $t$, as defined in Definition \ref{d:cross} and show that it is a cotriple on the category of functors from $\Cen n$ to ${\sc Ch(R)}$.  As discussed in the introduction to this paper, since ${\sc Ch(R)}$ is not a simplicial model category, we cannot use the analogous result (Theorem 3.8) in  [BJM].  
However, Lemma \ref{l:ifiber} gives us an explicit chain complex model for the iterated fiber.  We show below that one can work  with this model  to prove directly that $t$ is a cotriple in this setting.   The remainder of this section will be used to prove the next theorem.

\begin{thm}\label{t:tiscotriple}
There are natural transformations $\gamma: t\rightarrow {\rm id}$ and $\xi:t\rightarrow tt$ that make $(t,\gamma, \xi)$ a cotriple on ${\rm Fun}(\Cen n, {\sc Ch(R)})$.  
\end{thm}

To define $\gamma$ and $\xi$ and to make it easier to show that these maps satisfy the coassociativity and counitality conditions necessary for $t$ to be a cotriple, we  introduce some alternative notation for subsets of ${\bf n}$, ${\bf 2n}$, and ${\bf 3n}$ for a fixed $n\geq 1$.   
In some cases, it will be convenient to treat subsets of ${\bf n}$, ${\bf 2n}$, and ${\bf 3n}$ as matrices.  
\begin{defn}\label{n:matrix}
\begin{itemize}
\item[]
\item Let $M_{tn}$ be the set of $t\times n$ matrices whose entries are either $0$ or $1$.
\item Let $U$ be a subset of ${\bf n}$.   We will identify $U$ with the $1\times n$ matrix $[u_i]$ with 
\[
u_i=\begin{cases}1&\text{if}\  i\in U,\\
0&\text{if}\  i\notin U.\\
\end{cases}
\]
Note that under this notation, $|U|=\sum_{i=1}^nu_i$.  
\item A subset $W\subseteq {\bf 2n}$ will be represented by a $2\times n$ matrix $[w_{ij}]$ with 
\[
w_{1j}=\begin{cases} 1&\text{if}\  j\in W,\\
0&\text{if}\ j\notin W,\\
\end{cases}
\]
and 
\[
w_{2j}=\begin{cases} 1&\text{if}\  n+j\in W,\\
0&\text{if}\ n+j\notin W.\\
\end{cases}
\]
Again, we have $|W|=\sum _{j=1}^2\sum _{i=1}^n w_{ij}.$
\item Similarly, for $V\subseteq {\bf 3n}$, we identify $V$ with the $3\times n$ matrix $[v_{ij}]$ where
\[v_{ij}=
\begin{cases} 
1&\text{if}\ (i-1)n+j\in V, \\
0&\text{if}\ (i-1)n+j\notin V. \\
\end{cases}
\]
As before, $|V|=\sum_{i=1}^3\sum_{j=1}^{n}v_{ij}.$
\item For $U\subseteq {\bf n}$, let $M_{2n}(U)=\{W\in M_{2n}\ |\ w_{1j}+w_{2j}=u_j\  \mathrm{for\ all\ }j\}$ and $M_{3n}(U)=\{A\in M_{3n}\ |\ a_{1j}+a_{2j}+a_{3j}=u_j\ \mathrm{for\ all}\ j\}$.  

\end{itemize} 
\end{defn}

\begin{ex}  Let $n=2$.  For $U=\{1\}\subseteq {\bf 2}$, the matrix associated to $U$ is $[1\ 0]$ and $M_{2n}(U)$ consists of the matrices
\[
\left [\begin{matrix}1&0\\0&0\end{matrix}\right ], \left [\begin{matrix}0&0\\1&0\end{matrix}\right ].
\]
For $U=\{1,2\}\subseteq {\bf 2}$,  the matrix associated to $U$ is $[1 \ 1]$ and $M_{2n}(U)$ consists of the matrices
\[\left [\begin{matrix}1&0\\0&1\end{matrix}\right ], \left [\begin{matrix}0&1\\1&0\end{matrix}\right ], \left [\begin{matrix}1&1\\0&0\end{matrix}\right ], \left [\begin{matrix}0&0\\1&1\end{matrix}\right ].
\]
\end{ex}
In other cases, it will be convenient to view subsets of ${\bf 2n}$ as ordered pairs of subsets of ${\bf n}$, and subsets of ${\bf 3n}$ as ordered triples of subsets of ${\bf n}$.
\begin{rem}\label{n:tuple}
There is a one-to-one correspondence between subsets of ${\bf 2n}$ and elements of $\P({\bf n})\times \P({\bf n})$ given by 
\[
W\subseteq {\bf 2n}\mapsto (W_1, W_2)
\]
where $W_1=\{i\ |\ 1\leq i\leq n\ {\rm and}\ i\in W\}$ and $W_2=\{j-n\ | \ n+1\leq j\leq 2n\ {\rm and}\ j\in W\}$.  
Similarly, there is a one-to-one correspondence between subsets of ${\bf 3n}$ and elements of $\P({\bf n})^{\times 3}$ given by 
\[
S\subseteq {\bf 3n}\mapsto (S_1, S_2, S_3)
\]
where 
\[S_t=\{i-(t-1)n\ |\ (t-1)n+1\leq i\leq tn\ {\rm and}\ i\in S\}.  \]
\end{rem}

\begin{ex}
For $n=2$ and $W=\{1,2,4,5\}\subseteq {\bf 3n}$, $W_1=\{1,2\}$, $W_2=\{2\}$, and $W_3=\{1\}$.  
\end{ex}
With this, we can describe $ttF({\bf X})$ for a functor $F$ and $n$-tuple $\bf X$.  
The key to constructing the natural transformation $\xi:tF \to ttF$ is understanding $ttF$ as the iterated fiber of a $(2n)$-cube.  Recall from Lemma \ref{l:ifiber} that in degree $k$, 
\[
tF({\bf X})_k=\bigoplus _{T\subseteq {\bf n}}F({\bf X}(T))_{k+|T|}. 
\]
For an $n$-tuple ${\bf X}=(X_1, \dots, X_n)$, $ttF({\bf X})$ is the iterated fiber of the $n$-cube that assigns $tF({\bf X}(U))$ to  $U\subseteq {\bf n}$.  It follows immediately that
\[ ttF({\bf X})_k = \bigoplus_{U\subseteq {\bf n}} tF({\bf X}(U))_{k+|U|}.\]     But $tF({\bf X}(U))$ is itself the iterated fiber of an $n$-cube, and it is straightforward to show that it is the iterated fiber of the $n$-cube that assigns $F({\bf X}(U\cup V))$ to $V\subseteq {\bf n}$.  In this way, we see that $ttF({\bf X})$ is the iterated fiber of the $(2n)$-cube given by 
\[
W\subseteq {\bf 2n}\mapsto F({\bf X}(W_1 \cup W_2))
\]
where $W_1$ and $W_2$ are as defined in Remark \ref{n:tuple}.  
Under this correspondence, 
\[
ttF({\bf X})_k=\bigoplus _{(U,V)\in \P({\bf n})\times \P({\bf n})}F({\bf X}(U\cup V))_{k+|U|+|V|}.
\]

We now define the natural transformations $\gamma$ and $\xi$ of the cotriple $(t, \gamma, \xi)$.  
\begin{defn}
Let $F:{\Cen n}\rightarrow {\sc Ch(R)}$ and ${\bf X}=(X_1, \dots, X_n)$ be an $n$-tuple of objects in $\C^{\eta}$. 
\begin{itemize}
\item To define the natural transformation $\xi: tF({\bf X})\rightarrow ttF({\bf X})$, we produce maps for each $k$:  
\[
\bigoplus _{T\subseteq {\bf n}}F({\bf X}(T))_{k+|T|} \to \bigoplus _{(U,V)\in \P({\bf n})\times \P({\bf n})}F({\bf X}(U\cup V))_{k+|U|+|V|}.
\]
For a fixed $T\subseteq {\bf n}$ and $y\in F({\bf X}(T))_{k+|T|}$, the natural transformation $\xi$ sends $y$ to the sum of terms
\[
\mathop{\sum}_{ \substack{(V_1, V_2)\in P({\bf n})\times P({\bf n}), \\  V\in M_{2n}(T)}}(-1)^{{\mathrm {sgn}}(V)}y
,\] where ${\mathrm{sgn}}(V)$ is determined as follows.  
If $V=[v_{ij}]$,  then
\[
{\mathrm {sgn}}(V)=|\{i<j\ | \ v_{2i}=v_{1j}=1\}|.
\]
{ We note that for $V\in M_{2n}(T)$, $(-1)^{{\rm sgn}(V)}y$ is in the summand indexed by the pair $(V_1, V_2)\in \P({\bf n})\times \P({\bf n})$ corresponding to the subset of ${\bf 2n}$ associated with $V$. } 
\item The natural transformation $\gamma: tF({\bf X})\rightarrow F({\bf X})$ is given in degree $k$ by projection onto the summand indexed by $\emptyset$.  That is, 
\[
\gamma: tF({\bf X})_k=\bigoplus _{T\subseteq {\bf n}}F({\bf X}(T))_{k+|T|}\rightarrow F({\bf X}(\emptyset))_k=F(X_1, \dots, X_n).  
\]
\end{itemize} 
\end{defn}
Before proceeding, we need to make sure that these maps are chain maps.   
\begin{lem} \label{l:chainmap}  For a functor $F$ and $n$-tuple ${\bf X}$, $\gamma: tF({\bf X})\rightarrow F({\bf X})$ and $\xi: tF({\bf X})\rightarrow ttF({\bf X})$ are chain maps.
\end{lem}

\begin{proof}
The fact that $\gamma$ is a chain map follows directly from the definition.  The proof that $\xi$ is a chain map involves some bookkeeping. We use $d_t$ to denote the differential for $tF({\bf X})$ and $d_{tt}$ to denote the differential for $ttF({\bf X})$.   We must show that the diagram 
\[
\xymatrix{tF({\bf X})_k\ar[r]^{\xi}\ar[d]_{d_t}&ttF({\bf X})_k\ar[d]^{d_{tt}}\\
tF({\bf X})_{k-1}\ar[r]_{\xi}&ttF({\bf X})_{k-1}}
\]
commutes for all $k$.   

Recall that  
\[
ttF({\bf X})_k=\bigoplus _{(U,V)\in \P({\bf n})\times \P({\bf n})}F({\bf X}(U\cup V))_{k+|U|+|V|}.
\]

To understand $d_{tt}$, let $W\subseteq {\bf 2n}$ and  $W_1, W_2$ be the pair of subsets of ${\bf n}$ corresponding to $W$. Let  $y\in F({\bf X}(W_1\cup W_2))_{k+|W|}\subseteq ttF({\bf X})_k$.  From Lemma \ref{l:ifiber}, recall that $\sigma_i^T$ denotes the inclusion $T\to T\cup \{i\}$, and ${\mathrm {sgn}}(\sigma_i^T)=|\{s\in T\ |\ s>i\}|$.   The differential $d_{tt}(y)$ is the sum
\[
(-1)^{|W|}d(y)+\sum _{s\notin W}(-1)^{\mathrm{sgn}(\sigma ^W_s)+1}F({\bf X}(\tau _{s,W}))(y)
\]
where  the term indexed by $s\notin W$ is in the summand of $ttF({\bf X})_{k-1}$ indexed by $W\cup \{s\}$, and in terms of morphisms in $\P({\bf n})$, 
\[
\tau_{s,W}=\begin{cases} {\mathrm {id}}&\text{if}\ s+n\in W\ \text{or}\ s-n\in W,\\
\sigma _s^{W_1\cup 
W_2}&\text{if}\ s\leq n, s+n\notin W, \\
\sigma_{s-n}^{W_1\cup 
W_2}&\text{if}\ s>n, s-n\notin W.
\end{cases}
\]

Now consider $x\in F({\bf X}(U))_{k+|U|}\subseteq tF({\bf X})_k$.   Using $W$ to represent both a subset of ${\bf 2n}$ and its corresponding matrix, we have 
\begin{align}\label{dxi}
d_{tt}\xi(x)&=\sum _{W\in M_{2n}(U)}(-1)^{\mathrm{sgn}(W)}(-1)^{|W|}d(x)\\
&+\sum_{W\in M_{2n}(U)}\sum _{j\notin W}(-1)^{\mathrm{sgn}(W)}(-1)^{\mathrm{sgn}(\sigma_j^W)+1}F({\bf X}(\tau _{j,W}))(x),
\end{align}
and
\begin{align}\label{xid}
\xi d_t(x)&=\sum_{W\in M_{2n}(U)}(-1)^{\mathrm{sgn}(W)}(-1)^{|U|}d(x)\\
&+\sum _{i\notin U}\sum_{V\in M_{2n}(U\cup \{i\})}(-1)^{\mathrm{sgn}(\sigma_i^U)+1}(-1)^{\mathrm{sgn}(V)}F({\bf X}(\sigma_i^U))(x).
\end{align}
The first sums, in lines (3) and (5), are the same for both compositions since $|W|=|U|$ for a set $W$ corresponding to a matrix  in $M_{2n}(U)$.  We must show that the sums (4) and (6) are the same.  Expanding the sum in (4), one finds that it has more terms than the sum in  (6).   
%
{ The extra terms all correspond to sets $R=W\cup \{j\}$ where either $j+n$ or $j-n$ is in $W$.  These are also the terms where $\tau_{s,W}$ is the identity.  For any such $R$, one can show that there are two terms mapped into the summand indexed by $R$, one corresponding to the matrix in $M_{2n}(U)$ with a $1$ in the position corresponding to $j$ in the first row and the other with that $1$ in the second row.  For example, if $j+n\in W$, these matrices correspond to $W\cup \{j\}$ and $W'\cup \{j+n\}$ where $W'=(W-\{j+n\})\cup \{j\}$ and the  matrices are
\[W=\left[\begin{matrix}  x_{11}&\dots&x_{1j-1}&0&x_{1j+1}&\dots&x_{1n}\\
x_{21}&\dots&x_{2j-1}&1&x_{2j+1}&\dots&x_{2n}
\end{matrix} \right],\]
and
\[ W'=\left[\begin{matrix}  x_{11}&\dots&x_{1j-1}&1&x_{1j+1}&\dots&x_{1n}\\
x_{21}&\dots&x_{2j-1}&0&x_{2j+1}&\dots&x_{2n}
\end{matrix} \right].
\]  The signs cause the two terms to cancel one another.  The other terms in (4) are indexed by pairs $(W,j)$ where neither $j+n$ nor $j-n$ are in $W$. }   One can show that all such terms in (4) appear exactly once in (6) with the same sign.  
\end{proof}
Our next step is to show that $t$ is counital.  

\begin{lem}
For any functor $F$, the diagram of natural transformations commutes:
\[
\xymatrix{&tF\ar[dl]_{\rm id}\ar[dr]^{\rm id}\ar[d]_{\xi}& \\
tF&ttF\ar[l]^{t\gamma}\ar[r]_{\gamma_t}&tF.}
\]
\end{lem}
\begin {proof}
Recall 
that  
\[
ttF({\bf X})_k=\bigoplus _{(U,V)\in \P({\bf n})\times \P({\bf n})}F({\bf X}(U\cup V))_{k+|U|+|V|}.
\]
The map $t\gamma$ is projection of the summands indexed by pairs of the form 
$(\emptyset, V)$ onto the term indexed by $V$ in $tF({\bf X})_k$.   
Similarly, $\gamma _t$ is the projection of  summands indexed by pairs of the form $(U, \emptyset)$. 

To see what the compositions $\gamma_t\circ \xi$ and $t\gamma\circ \xi$ do to elements of $tF({\bf X})$, recall that 
\[
tF({\bf X})_k=\bigoplus _{T\subseteq {\bf n}} F({\bf X}(T))_{k+|T|},
\]
and let $y\in F({\bf X}(T))_{k+|T|}$.  The image of $y$ under $\xi$ is  
\[ \sum _{V\in M_{2n}(T)}(-1)^{{\mathrm {sgn}}(V)}y.
\]
We note that the only $V\in M_{2n}(T)$ that corresponds to a pair of the form $(U,\emptyset)$ in $\P({\bf n})\times \P({\bf n})$ is the matrix $T'=[t_{ij}]$ with $t_{2j}=0$ for all $j$ and $t_{1j}=1$ iff $j\in T$.  So the image of $y$ under the composition $\gamma_t\circ \xi$ is 
\[
(-1)^{{\mathrm {sgn}}(T')}y.
\]
But ${\mathrm{sgn}}(T')=0$, so $\gamma_t\circ\xi$ is the identity map.  In a similar fashion, we see that the only element of $M_{2n}(T)$ that corresponds to a pair of the form $(\emptyset, U)$ is the matrix $S=[s_{ij}]$ with $s_{1j}=0$ for all $j$ and $s_{2j}=1$ iff $j\in T$.  Again ${\mathrm{sgn}}(S)=0$ and it follows that $t\gamma\circ \xi$ is the identity.  

\end{proof}

Finally, we show that $t$ is coassociative.  
\begin{lem}
For any functor $F$, the diagram of natural transformations commutes:
\[
\xymatrix {tF\ar[d]_{\xi}\ar[r]^{\xi}&ttF\ar[d]^{t\xi}\\
ttF\ar[r]_{\xi_t}&tttF.}
\]
\end{lem}
\begin{proof}
To start, we note that, as was the case with $ttF({\bf X})$, we can view $tttF({\bf X})$ as the iterated fiber of the $(3n)$-cube that assigns to the set $W\subseteq {\bf 3n}$ the object
$F({\bf X}(W_1\cup W_2\cup W_3))$ where $(W_1, W_2, W_3)$ is the triple of subsets of ${\bf n}$ that correspond to $W$ as described in Remark \ref{n:tuple}.

As we will be working with iterations of $\xi$, it will be convenient to view subsets of ${\bf n}$, ${\bf 2n}$, and ${\bf 3n}$ as matrices, as described in Definition \ref{n:matrix}.   We introduce some additional notation for dealing with these matrices.  

\begin{itemize}
\item If $U$ is a $j\times n$ matrix and $1\leq l\leq j$, then $U_l$ is the $1\times n$ matrix whose only row is the $l$th row of $U$.  If $j\geq 2$ and $1\leq s<t\leq j$, then $U_{st}$ is the $2\times n$ matrix whose first row is the $s$th row of $U$ and second row is the $t$th row of $U$.
\item    If $U$ is a $j\times n$ matrix and $V$ is a $m\times n$ matrix, then $U|V$ will denote the $(j+m)\times n$ matrix whose first $j$ rows are the $j$ rows of $U$ and whose last $m$ rows are the rows of $V$. 
  
\item If $U$ and $V$ are $j\times n$ matrices, $U+V$ will denote their sum.  
 \end{itemize}

As in the previous proof, we will determine the effect of the two compositions in the diagram on a fixed $y\in tF({\bf X}(T))_{k+|T|}\subseteq tF({\bf X})_k$.  The composition $t\xi\circ \xi$ sends $y$ to the sum
\begin{equation}\label{e:topright}
\sum _{V\in M_{2n}(T)}\sum _{W\in M_{2n}(V_{1})}(-1)^{\mathrm{sgn}(W)+\mathrm{sgn}(V)}y=\sum _{C\in M_{3n}(T)}(-1)^{\mathrm{sgn}(C_{12})+\mathrm{sgn}((C_1+C_2)|C_3)}y
\end{equation}
and the composition $\xi_t\circ \xi$ sends $y$ to the sum
\begin{equation}\label{e:leftbottom}
\sum _{A\in M_{2n}(T)}\sum_{B\in M_{2n}(A_{2})}(-1)^{\mathrm{sgn}(B)+\mathrm{sgn}(A)}y=\sum _{D\in M_{3n}(T)}(-1)^{\mathrm{sgn}(D_{23})+\mathrm{sgn}(D_1|(D_2+D_3))}y.
\end{equation}
So, it suffices to show that for a given matrix $M\in M_{3n}(T)$, 
\[
\mathrm{sgn}(M_{23})+\mathrm{sgn}(M_1|M_2+M_3)=\mathrm{sgn}(M_{12})+\mathrm{sgn}(M_1+M_2|M_3),
\]
but this follows immediately from the definition of $\mathrm {sgn}$ for such matrices.  
\end{proof}
This completes the proof of Theorem \ref{t:tiscotriple}.  


\end{document}

\end{document}